\documentclass[12pt, english]{article}
\usepackage[T1]{fontenc}
\usepackage{amsmath}
\usepackage{amsthm}
\usepackage{amssymb}
\usepackage{enumerate}
\usepackage{color}
\numberwithin{equation}{section}
\makeatletter

\usepackage{amsthm}

\theoremstyle{plain}
\theoremstyle{remark}

\usepackage{fullpage}
\usepackage{makeidx}
\usepackage{amsfonts}
\usepackage{latexsym}
\usepackage{babel}
\usepackage{babel}
\makeatother

\def \1{{\bf 1}}
\def\C{{\mathbb{C}}}
\def\Z{{\mathbb{Z}}}
\def\R{{\mathbb{R}}}
\def\E{\mathbb{E}}

\def\gr{\mathrm{gr}}

\def\End{\mathrm{End}}
\def\g{\mathfrak{g}}
\def\wt{\mathrm{wt}}
\def\Aut{\mathrm{Aut}}
\newcommand{\F}{\mathbb{F}}
\newcommand{\K}{\mathbb{K}}

\def\<{\langle}
\def\>{\rangle}

\theoremstyle{definition}
\newtheorem{lemma}{Lemma}[section]
\newtheorem{theorem}[lemma]{Theorem}
\newtheorem{corollary}[lemma]{Corollary}
\newtheorem{proposition}[lemma]{Proposition}
\newtheorem{definition}[lemma]{Definition}

\newtheorem{example}[lemma]{Example}
\newtheorem{remark}[lemma]{Remark}

\usepackage{times}

\usepackage[blocks]{authblk}
	
\title{Endomorphism property of vertex operator algebras over arbitrary fields}

\author{Chao Yang}
\affil{ School of Mathematics,  Sichuan University,
	Chengdu 610064 China}

\author{Jianqi Liu}
\affil{Department of Mathematics, University of
California, Santa Cruz, CA 95064 USA}

\begin{document}
\maketitle
\abstract
In this paper, we study the endomorphism properties of vertex operator algebras over an arbitrary field $\F$, with  $\text{Char}(\F) \neq 2$.
Let $V$ be a strongly finitely generated vertex operator algebra over $\F$, and $M$ be an irreducible admissible $V$-module.
We prove that every element in $\text{End}_V(M)$ is algebraic over $\F$ and that $\text{End}_V(M)$ is also finite-dimensional.
As an application, we prove Schur's lemma for strongly finitely generated vertex operator algebras over arbitrary algebraically closed fields, and we give a test for absolute irreducibility of $V$-modules.


\section{Introduction}

The study of the endomorphism properties for associative algebras has a long history and rich results(see \cite{D, DS,MR,Q} etc.).
If the associative algebra $A$ is defined over an algebraically closed field
and $S$ is a finite dimensional simple $A$-module, the famous Schur's Lemma tells us that $\End_{A}(S)$ is one dimensional.
However, when the base field $\F$ is not algebraically closed, or $\text{dim}_F(S)$ is not finite,
$\End_{A}(S)$ is not so easy to characterize. One of the most elegant results, in this case, is Quillen's Theorem \cite{Q}, which says that each element in $\text{End}_A(S)$ is algebraic over $\F$, provided that $A$ is a filtered algebra whose associated graded algebra $\text{gr}(A)$ is finitely generated and commutative.

In some special cases, the endomorphism ring $\text{End}_A(S)$ is finite-dimensional. For example, if $A=U(\mathfrak{g})$ is the universal enveloping algebra of a finite dimensional lie algebra $\mathfrak{g}$ over a field $\F$ with $\text{char}(\F)=0$, it is shown in \cite{D} that $\text{End}_A(S)$ is finite dimensional; if $A$ is a constructible algebra over an arbitrary field $\F$, it is shown in \cite{MR} that $\text{End}_A(S)$ is also finite dimensional. Moreover, if $S$ is a topologically irreducible admissible Banach space representation
of a $\mathbb{Q}_p$-analytic group $G$ over $\F$, where $\F$ is a finite extension of $\mathbb{Q}_p$, then it is shown in \cite{DS} that the algebra $\rm{End}_{\F[G]}(S)$ of continuous $G$-equivariant
endomorphisms of $S$ is a finite dimensional division algebra over $\F$. An associative algebra over $\F$ such that $\End_{A}(S)$ is finite-dimensional for all simple modules $S$ is said to admit the (strong) endomorphism property.

In the theory of vertex operator algebras (VOAs) (\cite{DLM1,LL} etc.), when the base field is $\C$, the ``infinite dimensional Schur's Lemma" due to Dixmier \cite{D} which says that $\End_{A}(M)=\C \text{Id}_{M}$ if $M$ is of countable dimension is commonly used in the proof of the semisimplicity of the Hamiltonian $L(0)$ on an irreducible admissible module $M$ over a vertex operator algebra $V$ (cf. \cite{DLM,DLM1,Z}.) In recent years, the vertex operator algebras defined over an arbitrary field $\F$ has brought much attention. A series of studies on vertex operator algebras over arbitrary fields can be found in (\cite{DR1, DR2, R, JLM, LM1,LM2} etc.) However, since the usual Schur's lemma no longer holds when $\F$ is not uncountable, a natural question is what kind of endomorphism properties are possessed by a vertex operator algebra over an arbitrary field $\F$. The following Theorem is our main result.

\begin{theorem} \label{1main1}
Let $V$ be a strongly finitely generated vertex operator algebra over an arbitrary field $\F$,
and $M$ be an irreducible admissible $V$-module. Then
\begin{enumerate}[{(1)}]

\item Every element in $\text{End}_V(M)$ is algebraic over $\F.$
 Moreover,  $\text{End}_V(M)$ is also finite-dimensional over $\F$.

\item (Schur's lemma) If $\F$ is an algebraic closed field, then $\text{End}_V(M) \cong \F$, and
 there exists a $\lambda \in \F$ such that $L(0)|_{M(n)}=\lambda +n$ for all $n$.

\item  $M$ is an absolutely irreducible $V$-module if and only if  $\text{End}_V(M) \cong \F.$
\end{enumerate}
\end{theorem}

In this theorem, part (1) can be viewed as the analog of Quillen's theorem in the context of the vertex operator algebras. Schur's lemma and the characterization of the absolutely irreducible modules are natural corollaries of Quillen's theorem in (1). Although we imposed the condition that $V$ is strongly finitely generated in Theorem \ref{1main1}, we note that this is not a strong condition. In fact, almost all the classic examples of vertex operator algebras over an arbitrary field, including the irrational ones, like vacuum module VOA $V_{\hat{\g}}(k,0)$ and the (universal) Virasoro VOA $\bar{V}(c,0)$ \cite{FZ}, are strongly finitely generated \cite{DR2,JLM,LM1}. Furthermore, for a strongly finitely generated VOA $V$, many examples indicate that the orbifold $V^G$ of a reductive group $G\leq \Aut(V)$ is also strongly finitely generated (see \cite{LA,CL}.)
With Theorem \ref{1main1}, we can prove that the conjecture proposed in \cite{DR1}, which states that the Hamiltonian $L(0)$ acts semi-simply on an irreducible admissible $M$ without the uncountable condition for $\F$ is indeed true. 

To prove Theorem \ref{1main1}, we use the (generalized) Zhu's algebra $A_n(V)$ defined in \cite{DLM,R,Z} and transform the study of $\End_{V}(M)$ into
the study of the endomorphism property of an $A_n(V)$-modules. We will also use the fact that there exists an epimorphism between Zhu's $C_{2}$-algebra $R(V)=V/C_{2}(V)$ and $\text{gr}A(V)$ \cite{A,CJL,Liu}, where $\gr A(V)$ is the associated graded algebra of the level filtration on $A(V)$, and transform the strong generated property of $V$ into a finite generation property of $\gr A(V)$.
More precisely, we have the following results toward the proof of Theorem \ref{1main1}:

\begin{theorem}\label{1main2}
\begin{enumerate}[{(1)}]
\item  Let $V$ be a vertex operator algebra over a field $\F$, and let $M=\bigoplus_{i \geq 0}M(i)$ be an irreducible admissible $V$-module.
Then $\text{End}_V(M) \cong \text{End}_{A(V)}(M(0))$.

\item  Let $V$ be a strongly finitely generated vertex operator algebra over a field $\F$.
Then $A(V)$ is a filtered algebra such that the associated graded algebra $\text{gr}A(V)$ is a finitely generated commutative algebra.

\item Let $A$ be a filtered algebra such that the associated graded algebra $\text{gr}(A)$ is a finitely generated commutative algebra.
Then for each simple $A$-module $S$, every element in $\text{End}_A(S)$ is algebraic over $\F$.
Moreover,  $\text{End}_V(S)$ is finite dimensional.
\end{enumerate}
\end{theorem}

This paper is organized as follows. Section 2 reviews the basics of vertex operator algebras over an arbitrary field. We introduce some notations and prove some basic results about scalar extensions of vertex operator algebra in section 3.
In section 4, we first recall the definition and properties of the generalized Zhu's algebra $A_n(V)$ in \cite{DLM, R}, and then prove an intermediate result about the endomorphism property of the finite-type filtered algebras.
In section 5 and section 6, we prove the main results of Theorem \ref{1main1} and Theorem \ref{1main2}.
Some applications of the main Theorems on the scalar extensions are discussed in section 7.

The authors are grateful to our supervisor Professor  Chongying Dong for valuable advice.

{\bf  Convention}: All the fields in this paper have characterization {\em not} equal to $2$.

\section{Basics}
In this section, we recall the basics of vertex operator algebras and modules \cite{DR1,LL,LM1,LM2}
over an arbitrary field $\F$ with $\text{char}(\F) \neq 2$.

\begin{definition}\label{df2.1}
A vertex algebra $(V, Y(-, z) , {\bf 1} )$ over a field $\F$ is a vector space equipped with
a linear map:
\begin{align*}
		Y(-, z):  V & \to (\mbox{End}\,V)[[z,z^{-1}]] ,\\
		& v\mapsto Y(v,z)=\sum_{n\in{\Z}}v_nz^{-n-1}\ \ \ \  (v_n\in
		\mbox{End}\,(V)),\nonumber
\end{align*}
and with a distinguished vector ${\bf 1} \in V$, satisfying the following conditions for any $u,v\in V$ and $m,n\in \Z$:
\begin{align*} \label{0a4}
		& u_nv=0\ \ \ \ \ {\rm for}\ \  n\ \ {\rm sufficiently\ large};  \\
		& Y({\bf 1},z)=Id_{V};  \\
		& Y(v,z){\bf 1}\in V[[z]]\ \ \ {\rm and}\ \ \
        \lim_{z\to 0}Y(v,z){\bf 1}=v;
	\end{align*}
and the Jacobi identity:
		\begin{align*}
			& \displaystyle{z^{-1}_0\delta\left(\frac{z_1-z_2}{z_0}\right)
				Y(u,z_1)Y(v,z_2)-z^{-1}_0\delta\left(\frac{z_2-z_1}{-z_0}\right)
				Y(v,z_2)Y(u,z_1)}\\
			& \displaystyle{=z_2^{-1}\delta
				\left(\frac{z_1-z_0}{z_2}\right)
				Y(Y(u,z_0)v,z_2)}.
		\end{align*}
\end{definition}

We denote the set of non-negative integers by $\Z_+$.

\begin{definition}\label{df2.2}
A vertex operator algebra $(V, Y(\ , z), {\bf 1}, \omega)$ over a field $\F$ is a $\Z_+$-graded vertex algebra $(V, Y(-, z), {\bf 1})$:
$$V=\bigoplus_{n \in \Z_{+}}V_n,$$
satisfying $\text{dim}V_n < \infty$ for all $n\in \Z_{+}$, ${\bf 1} \in V_0$, ${\omega} \in V_2$,
and the following conditions holds for any $u,v\in V$ and $m,n,s,t \in \Z$:
\begin{align*}
  & u_nv \in V_{s+t-n-1} \ \text{for} \ u\in V_s, v\in V_t; \\
  & [L(m),L(n)]=(m-n)L(m+n)+\frac{1}{12}(m^3-m)\delta_{m+n,0}c ;\\
  & \frac{d}{dz}Y(v,z)=Y(L(-1)v,z);\\
  & L(0)|_{V_n}=n,
\end{align*}

where $L(m)=\omega_{ m+1}$ for each $m\in \Z$,  that is,
	$Y(\omega,z)=\sum_{n\in\Z}L(n)z^{-n-2}$, and $c\in \F.$
	
	We denote a vertex operator algebra $(V,Y(-,z),\1,\omega)$ over a field $\F$ by $V$ for short. If $u \in V_n$, we call $n$ the degree of $u$ and write $\text{deg}u=n$.
\end{definition}

\begin{remark}
If $\text{char}(\F)=0$, the assumption $u_nv \in V_{s+t-n-1}$ in Definition \ref{df2.2} is a consequence of the
other axioms. In the case that $F$ is a field of finite characteristic $p$, $L(0)$ has only $p$
distinct eigenvalues on $V$, and this assumption is necessary.
\end{remark}


\begin{definition}
Let $V$ be a vertex operator algebra over a field $\F$.
A weak $V$-module $M$ is a vector space equipped with a linear map:
\begin{equation*}
		\begin{split}
			Y_M(-,z): V&\to {\rm End(M)}[[z,z^{-1}]]\\
			v&\mapsto\displaystyle{ Y_M(v,z)=\sum_{n\in \Z}v_nz^{-n-1}\ \ \ (v_n\in
				{\rm End(M)}}),
		\end{split}
	\end{equation*}
which satisfies the following properties:  for any $u, v\in V,$ and $w\in M$,
\begin{eqnarray*}
		& &u_lw=0~~~
		\mbox{for}~~~ l\gg 0;\label{vlw0}\\
		& &Y_M({\mathbf 1},z)=Id_M;\label{vacuum}
\end{eqnarray*}
\begin{align*}
			&\displaystyle{z^{-1}_0\delta\left(\frac{z_1-z_2}{z_0}\right)
				Y_M(u,z_1)Y_M(v,z_2)-z^{-1}_0\delta\left(\frac{z_2-z_1}{-z_0}\right)
				Y_M(v,z_2)Y_M(u,z_1)}\\
			&\displaystyle{=z_2^{-1}\delta\left(\frac{z_1-z_0}{z_2}\right)
				Y_M(Y(u,z_0)v,z_2)}.
\end{align*}
\end{definition}

\begin{definition}\label{df2.5}
An admissible $V$-module is a $\Z_+$-graded weak $V$-module $M$:
$$M=\bigoplus_{n \in \Z_+}M(n),$$
satisfying
$$u_n M(t) \in M(s+t-n-1)$$
for $u \in V_s$, $n\in \Z$ and $t \in \Z_+$.
\end{definition}

We also need the concept of strongly finitely generated vertex operator algebra (cf. \cite{L1}):
\begin{definition}
Let $V$ be a vertex operator algebra over a field ${\F}$, and let $U\subset V$ be a subset. $V$ is called strongly generated by $U$ if $V$ is spanned by elements of the form:
$$u^{1}_{-n_{1}}\dots u^{r}_{-n_{r}}{\bf 1},$$
where $u^{1},\dots,u^{r} \in U$ for $r \geq 0$, and $n_{i}\geq 1$ for all $i$.
If $V$ is strongly generated by a finite dimensional subspace, then $V$ is called strongly finitely generated.
\end{definition}

\begin{definition}
Let $V$ be a vertex operator algebra over a field $\F$.
\begin{enumerate}[{(1)}]

\item  $V$ is called rational if the admissible $V$-module category is semi-simple.

\item $V$ is called $C_2$-cofinite if $R(V)=V/C_2(V)$ is finite dimensional \cite{LM1}, where
$$C_2(V)=\text{span}\{v_{-n-2}v \ | \ u, v\in V, n\in \Z_+\}.$$
\end{enumerate}
\end{definition}

\begin{definition}\label{df2.8} (\cite{MR})
Let $A$ be an associative algebra over a field $\F$, $A$ is said to have the endomorphism property if for each simple $A$-module $S$, every element in $\text{End}_A(S)$ is algebraic over $\F$. $A$ is said to have the strong endomorphism property if for each simple $A$-module $S$,
the endomorphism algebra $\text{End}_A(S)$ is finite dimensional.
\end{definition}
We introduce a similar definition of endomorphism properties for vertex operator algebras over a field $\F$:
\begin{definition}\label{df2.9}
Let $V$ be a vertex operator algebra over a field $\F$. $V$ is said to have the {\bf endomorphism property} if for each irreducible admissible $V$-module $M$, every element in $\text{End}_V(M)$ is algebraic over $\F.$

$V$ is said to have the {\bf strong endomorphism property} if for each irreducible admissible $V$-module $M$, the endomorphism algebra $\text{End}_V(M)$ is finite dimensional over $\F$.
\end{definition}

\section{Scalar extensions}

In this section, we introduce some notations and prove some basic properties of scalar extension of vertex operator algebras, which will be used later.

Let $A$ be an associative algebra over a field $\F$, and $M$ be a (left) $A$-module. Let $\K/\F$ be a field extension. Define $A^{\K}=\K \otimes_{\F} A$ and $M^{\K}=\K \otimes_{\F} M$,
then $A^{\K}$ is an associative algebra over $\K$, and $M^{\K}$ is a (left) $A^{\K}$-module.


\begin{definition} (\cite{D})
Let $A$ be an associative algebra  over a field $\F$.  An simple $A$-module $S$ is called absolutely simple if $S^{\K}$ is an simple $A^{\K}$-module for any field extension $\K$ of $\F$.
\end{definition}

\begin{definition} (\cite{B})
A noetherian algebra over a field $\F$ is called stably noetherian if $A^{\K}$ is also noetherian for any field extension $\K$ of $\F$.
\end{definition}

The following Lemma is well-known (see \cite{Lang}):

\begin{lemma} \label{ext-semi}
Let $A$ be a finite dimensional semi-simple $\F$-algebra.
Then for any finite separable field extension $\K/\F$, the $\K$-algebra $A^{\K}$ is also semi-simple.
\end{lemma}

 Let $ \K/\F$ be a field extension, and $(V, Y(-,z), {\bf 1}, \omega)$ be a vertex operator algebra over  $\F$. Let $(M, Y_M(- ,z))$ be an admissible $V$-module. We write:
 $$V^{\K}=\K \otimes_{\F} V, ~~ {\bf 1}^{\K}=1 \otimes {\bf 1},~~ \omega^{\K}=1 \otimes \omega, $$
 $$M^{\K}=\K \otimes_{\F} M=\bigoplus_{n=0}^\infty \K\otimes_\F M(n)=\bigoplus_{n=0}^\infty M(n)^\K,$$
 then define $Y^\K: V^\K \rightarrow \End{V^\K}[[z,z^{-1}]]$ and $Y_{M^\K}: V^\K\rightarrow \End(M^\K)[[z,z^{-1}]]$ by
 \begin{equation}\label{3.1}
   Y^{\K}(\lambda \otimes a , z)(\mu \otimes b):=\lambda \mu\otimes Y(a, z)b,
 \end{equation}
 \begin{equation}\label{3.2}
  Y_{M^{\K}}(\lambda \otimes a , z)(\mu\otimes u):=\lambda\mu \otimes Y_M(a, z)u,
 \end{equation}
 for all $a,b\in V,$ $\lambda, \mu \in \K$, and $u\in M$. Then it is easy to check that $(V^{\K}, Y^{\K}(-,z), {\bf 1}^{\K}, \omega^{\K})$ is a vertex operator algebra over $\K$,
and $(M^{\K}, Y_{M^{\K}}(- ,z))$  is a $V^{\K}$-module.

\begin{proposition} \label{exstf}
	Let $V$ be a vertex operator algebra over a field $\F$, and let $\K/ \F$ be a field extension. Then we have:
	\begin{enumerate}[{(1)}]
	\item $V$ is $C_2$-cofinite if and only if $V^{\K}$ is $C_2$-cofinite.
	
	\item $V$ is strongly finitely generated if and only if $V^{\K}$ is strongly finitely generated.
	\end{enumerate}
\end{proposition}
\begin{proof}
	The proof is straightforward from the definition.
\end{proof}

We can similarly define the absolute irreducibility for irreducible admissible modules of vertex operator algebras over arbitrary field.

\begin{definition}
Let $V$ be a vertex operator algebra over a field $\F$.  An irreducible admissible $V$-module $M$ is called {\bf absolutely irreducible} if $M^{\K}$ is an irreducible admissible $V^{\K}$-module for any field extension $\K/\F$.
\end{definition}


\begin{lemma} \label{lm3.6}
Let $V$ be a vertex operator algebra over a field $\F$,
and $M$ be an irreducible admissible $V$-module.
Let $\K/\F$ be a field extension.
Then $M^{\K}$ is an irreducible admissible $V^{\K}$-module  if and only if $M(0)^{\K}$ is a simple $A(V)^{\K}$-module.
\end{lemma}
\begin{proof}
Clearly, if $M^\K$ is irreducible then $M(0)^\K$ is a simple $A(V)^\K$-module.

Conversely,  suppose $M(0)^{\K}$ is a simple $A(V)^{\K}$-module, and
$N$ is an admissible submodule of $M^{\K}$ such that $0 \lneq N \lneq M^{\K}$.
Since $M(0)^{\K}$ is simple and $M^{\K}$ is generated by $M(0)^{\K}$, we have $N \cap M(0)^{\K}=0.$
Hence there exists some $s >0$ such that the bottom level $N(0)$ of $N$ is contained in $ M^{\K}(s)$.
Take a nonzero element $k_1 \otimes m_1+\cdots + k_p \otimes m_p \in N(0) \ ( k_i \in \K,\ 0 \neq  m_i \in M(s), \ p>0)$
such that $ k_1, \cdots , k_p $ are linearly independence over $\F$.
For any $v \in V$ and $i\geq 0$, by \eqref{3.2} we have:
\begin{align*}
0&=(1\otimes v)_{\wt(v)+i}(k_1 \otimes m_1+\cdots + k_p \otimes m_p )\\
&=k_1 \otimes (v_{\wt(v)+i}m_1)+ \cdots + k_p \otimes (v_{\wt (v)+i}m_p).
\end{align*}
Since $ k_1, \dots ,k_p $ are linearly independence over $\F$,
we have $v_{\wt(v)+i}m_1=0$ for any $v \in V$ and $i\geq 0$.
Let $\<m_1\>$ be the $V$-submodule of $M$ generated by $m_1$. Then it follows from an easy induction that $\<m_1\> \cap M(0)=0,$ and so $\<m_1\>$ is a proper submodule of $M$.
This contradicts the assumption that $M$ is irreducible. Hence $M^{\K}$ is irreducible, and the proof is complete.
\end{proof}

The following corollaries are immediate consequences of Lemma \ref{lm3.6}.
\begin{corollary} \label{ext-simple2}
Let $V$  be a simple vertex operator algebra over a field $\F$, with $V_0=\F{\bf 1}$, and let $\K / \F$ be a field extension. Then $V^{\K}$ is also a simple vertex operator algebra over a field $\K$.
\end{corollary}

\begin{corollary}
Let $V$ be a vertex operator algebra over a field $\F$.
If  $M$ is  an irreducible admissible $V$-module such that $M(0) \cong \F$,
then $M$ is absolutely irreducible.
\end{corollary}


\section{Associative algebras $A_n(V)$ and Quillen's Theorem}

Our goal in this section is to give a characterization of $\End_V(M)$, using the $A_n(V)$-theory of vertex operator algebras (over an arbitrary field) in \cite{DLM,R}, as a preparation of the main theorem of the paper.	We will also recall Quillen's theorem \cite{Q} for finite type filtered algebras and prove some preparation Lemmas for the following sections.

We first review the $A_n(V)$-theory of vertex operator algebras over arbitrary fields.

Let $V$ be a vertex operator algebra over a field $\F$, and $n$ be a nonnegative integer.
For any homogeneous elements $u\in V$, $v \in V$, and $s,t \in \Z_+$ with $t \geq s$, define
$$u \circ_{n,t}^{s}v=\text{Res}_{z}Y(u,z)v\frac{(1+z)^{\wt u+n+s}}{z^{2n+2+t}},$$
and
$$u*_nv=\sum_{m=0}^{\infty}(-1)^m{m+n \choose n}\text{Res}_zY(u,z)v\frac{(1+z)^{\wt u+n}}{z^{n+m+1}}.$$

Extend $\circ_{n,t}^{s}$ and $*_n$ bilinearly onto $V$. Let $O_{n}(V)$ be the span of elements  $u\circ_{n,t}^{s}v$ and $L(-1)u+L(0)u$, for all $u,v \in V$, and $s,t \in \Z_+$ with $t \geq s$.
Define $A_{n}(V)$ to be the quotient space $V/O_{n}(V)$.
In the case $n=0$, $A_0(V)$ is just the usual Zhu's algebra $A(V)$ (over field $\F$)
studied in \cite{DR1,Z}.

\begin{remark}
In the case $\text{Char}(\F)=0$, it is proved in \cite{DLM} that $O_n(V)$ is spanned by
$u \circ_{n,0}^0v$ and  $L(-1)u+L(0)u$, for all $u,v \in V$.
But for an arbitrary field $\F$, it is necessary (cf.\cite{R}) to include the elements $u \circ_{n,t}^{s}v$ in the definition of $O_n(V)$.
The notion $u \circ_{n,t}^{s}v$ comes from \cite{DJ}.
\end{remark}

For $u \in V$, we write $[u]=u+O_n(V) \in A_n(V)$. For a homogeneous element $u \in V$, we write $o(u)$ for the operator $u_{\wt u-1}$ on any weak $V$-module $M$.


\begin{theorem}\label{Anv1}(\cite{DLM,R})
Let $V$ be a vertex operator algebra over a field $\F$, and let $M$ be an admissible $V$-module. We have:
\begin{enumerate}[{(1)}]
\item For any $n \geq 0$, $(A_n(V), *_n, {\bf 1}+O_n(V))$ is an associative algebra.
\item For any $n \geq s \geq 0$,  the identity map on $V$ induces an onto algebra
homomorphism from $A_n(V)$ to $A_s(V)$.
\item For any $n \geq s \geq 0$, $M(s)$ is an $A_n(V)$-module such that $[u]\in A_{n}(V)$ acts as $o(u)$ for homogeneous $u \in V$.
\item If $V$ is rational, then $A_n(V)$ is a semi-simple associative algebra for any $n \geq 0$.
\item If $V$ is simple, then $V$ is rational if and only if  $A_n(V)$ is a semi-simple associative algebra for any $n \geq 0$.
\end{enumerate}

\end{theorem}

\begin{remark}
In \cite{DR1,R}, the field $\F$ is assumed to be uncountable and algebraically closed, then by using the infinite-dimensional Schur's Lemma, one can conclude that $A_{n}(V)$ is not only a semi-simple algebra over $\F$ but finite-dimensional as well.

However, since we do not assume that F is uncountable,
it is not obvious whether or not $A_n(V)$ is finite-dimensional in parts (4) and (5) of theorem \ref{Anv1}.
\end{remark}

\begin{theorem} \label{Anv2}(\cite{DLM,R})
Let $V$ be a vertex operator algebra over a field $\F$, and let $U$ be an $A_n(V)$-module.
Then there exists an admissible $V$-module $\mathcal{M}_n(U)$ such that
\begin{enumerate}[{(1)}]
\item $\mathcal{M}_n(U)$ is generated by $\mathcal{M}_n(U)(n)=U$, and it has the following universal property:
for any admissible $V$-module $M$ and any $A_n(V)$-morphism $\phi: U \rightarrow M(n)$,
there exists a unique $V$-morphism $\overline{\phi}: \mathcal{M}_n(U) \rightarrow M$ such that $\overline{\phi} |_U=\phi$.
\item $\mathcal{M}_n(U)$ has a unique maximal admissible $V$-submodule $\mathcal{J}_n(U)$, with the property that $\mathcal{J}_n(U) \cap U=0.$
Set $L_n(U)=\mathcal{M}_n(U)/\mathcal{J}_n(U)$. Then $L_n(U)$ is an admissible $V$-module generated by $L_n(U)(n)=U$.
\item If $U$ is a simple $A_n(V)$-module, then $L_n(U)$ is an irreducible $V$-module.
\item If $M=\bigoplus_{i=0}^\infty M(i)$ is an irreducible admissible $V$-module,
then $L_n(M(n))$ is irreducible, and  $L_n(M(n)) \cong M$.
\end{enumerate}
\end{theorem}

\begin{remark}
We note that $L_n(U)(0)$ might be zero if $U$ is also a module over $A_{n-1}(V)$.
\end{remark}
Similar to Proposition \ref{exstf}, $A_n(V)$ satisfies the following properties:
\begin{proposition} \label{Anvex}
Let $V$ be a vertex operator algebra over a field $\F$, and let $\K /\F$ be a field extension.
Then for any $n \in \Z_+$, we have $A_n(V^{\K}) \cong (A_n(V))^{\K}.$
\end{proposition}

\begin{corollary} \label{Anv3}
Let $V$ be a vertex operator algebra over a field $\F$. If $M=\bigoplus_{i \geq 0}M(i)$ is an irreducible admissible $V$-module,
then $\text{End}_V(M) \cong \text{End}_{A_n(V)}(M(n))$ for any $n \in \Z_+.$
\end{corollary}
\begin{proof}
Define a linear map $\varphi: \text{End}_V(M) \rightarrow \text{End}_{A_{n}(V)}(M(n)),$ $\varphi(F):=F|_{M(n)}$,
for any $F \in \text{End}_V(M)$. Since $M$ is irreducible, it is easy to see that $\varphi$ is injective.

Let $f\in \text{End}_{A_{n}(V)}(M(n))$, by Theorem \ref{Anv2} (1), there exists a $V$-homomorphism:
 $$\overline{f}:\mathcal{M}_n(M(n))\rightarrow M$$ such that $\overline{f}|_{M(n)}=f$.
Since $M$ is irreducible, we have $\text{Ker}\overline{f}=\mathcal{J}(M(n))$ by the definition of $\mathcal{J}(M(n))$.
Hence the $\overline{f}$ induces a  $V$-homomorphism
$\widetilde{f}: M \cong L_n(M(n)) \rightarrow M$ such that $\varphi(\widetilde{f})=\widetilde{f}|_{M(n)}=f$.
This shows that $\varphi$ is also surjective, and the proof is complete.
\end{proof}

Next, we recall Quillen's theorem for finite-type filtered algebras and prove some related results under field extensions.

\begin{definition}
An associative algebra $A$ over a field $\F$ is called filtered if
it is endowed with an increasing filtration:
$0=F_{-1}A \subset F_0A \subset F_1A \subset \cdots $ by $\F$-subspaces such that
\begin{enumerate}[{(1)}]
	\item $ 1_A \in F_0A,$
\item $ \ F_pA \cdot F_qA \subset F_{p+q}A$  for all $p,q \geq 0,$
\item $  A=\bigcup_{p \geq 0} F_pA.$
\end{enumerate}

Let $A$ be a filtered algebra. Set $\text{gr} A=\oplus_{n= 0}^\infty F_nA/F_{n-1}A=\bigoplus_{n=0}^\infty (\text{gr}A)_n $.
The multiplication in $A$ gives rise to a well-defined product:
$$(F_pA/F_{p-1}A) \times (F_qA/F_{q-1}A) \rightarrow F_{p+q}A/F_{p+q-1}A, \ \   (\bar{a},\bar{b})\mapsto \overline{a\cdot b},$$
for all $a,b\in V$, making $\text{gr}A$ a graded associative algebra.
\end{definition}

\begin{definition}
A filtered algebra $A$ over a field $\F$ is called finite-type if $\text{gr}A$ is a finitely generated commutative algebra.
\end{definition}

\begin{lemma} \label{quillen}  (Quillen's theorem)
Let $A$ be a finite-type filtered algebra over a field $\F$. Then $A$ has the endomorphism property. i.e., if $M$ is a simple $A$-module, then every element in $\text{End}_A(M)$ is algebraic over $\F$.
\end{lemma}

The following Proposition, which is a stronger version of Quillen's theorem, might be well-known, but
we cannot find any reference to this version. So we present the proof for it in the rest of this section.

\begin{proposition}\label{finite}
Let $A$ be a finite-type filtered algebra over $\F$. Then $A$ has the strong endomorphism property. i.e.,
if $M$ is a simple $A$-module, then $\text{dim}_{\F}\text{End}_A(M) < \infty.$
\end{proposition}
\begin{remark}
In the case when $\text{Char}(\F)=0$ and $A=U(\mathfrak{g})$ the universal enveloping algebra of a finite dimensional Lie algebra $\mathfrak{g}$ over $\F$,
Proposition \ref{finite} is proved in \cite[Proposition 2.6.9]{D}.
On the other hand, when $\F$ is of arbitrary characteristic and $A$ is a constructible algebra over $\F$, Proposition \ref{finite} is proved in \cite[Theorem 9.5.5]{MR}. Note that $U(\mathfrak{g})$ is a constructible algebra when $\mathfrak{g}$ is a finite-dimensional Lie algebra over $\F$. 

In our case, we cannot assume that $\text{gr}A$ is generated by $F_1A/F_0A$, since we will let $A$ be Zhu's algebra $A(V)$ later. So $A$ might not be a constructible algebra.
But we can still prove Proposition \ref{finite} using some techniques in \cite{MR}. 
\end{remark}

\begin{lemma} \label{division-ring}  (\cite{MR})
Let $D$ be a division ring whose center is $\K$, and let $\mathbb{E}$ be a maximal subfield of $D$.
If $\text{dim}_{\K}(\mathbb{E}) < \infty$, then $\text{dim}_{\mathbb{E}}(D) < \infty$ and $\text{dim}_{\K}(D) <\infty.$
\end{lemma}

The proof of the following Lemma is fairly standard (see \cite{MR}).
\begin{lemma} \label{Noteherian}
If $A$ is a filtered algebra such that $\text{gr}(A)$ is left noetherian, then $A$ is left noteherian.
\end{lemma}

The next Lemma shows that the finite-type property and the notherianess are preserved under field extensions.

\begin{lemma} \label{Noteherian2}
 Let $A$ be a finite-type filtered algebra over a field $\F$. Let $\K/\F$ be a field extension. Then
\begin{enumerate}[{(1)}]
 \item $A^{\K}$ is also a finite-type filtered algebra.
\item $A^{\K}$ is noetherian. In particular, $A$ is  noetherian.
 \end{enumerate}
\end{lemma}

\begin{proof}
Let  $0=F_{-1}A \subset F_0A \subset F_1A \subset \cdots $ be a filtration of $A$.
Then $$0=\K \otimes F_{-1}A \subset \K \otimes F_0A \subset \K \otimes F_1A \subset \cdots $$
is a filtration of $A^{\K}$, here $\otimes$ means $\otimes_{\F}$.
It is clear that $\text{gr}(A^{\K}) \cong \K \otimes_{\F} \text{gr}(A)$.
Hence $\text{gr}(A^{\K})$ is a finitely generated commutative algebra.
Consequently, $A^{\K}$ is a finite-type filtered algebra and $\text{gr}A^{\K}$ is left noetherian.
 By Lemma  \ref{Noteherian}, $A^{\K}$ is left noetherian.

 To prove $A^{\K}$ is right noetherian, we consider the opposite ring $(A^{\K})^{\text{op}}$ of $A^{\K}$.
Note that $A^{\K}$ is right noetherian if and only if $(A^{\K})^{\text{op}}$ is left noetherian.
It is easy to verify that $(F_nA^{\K})_{n=0}^{\infty}$ is also a filtration of $(A^{\K})^{\text{op}}$, where $F_nA^{\K}=K \otimes_{\F} F_n(A)$ for each $n\geq 0$.
Since $\text{gr}A^{\K}$ is commutative,
 $\text{gr}((A^{\K})^{\text{op}}) \cong \text{gr}A^{\K}$ is a finitely generated commutative algebra.
Therefore, $(A^{\K})^{\text{op}}$ is left noetherian, and so $A^{\K}$ is right noetherian. The proof is complete.
\end{proof}

The following lemma is the key to prove Proposition \ref{finite}.

\begin{lemma} \label{f-extension} (\cite{MR})
Let $A$ be a $\F$-algebra, $M$ be a finite generated $A$-module, and $\K$ be a subfield of $\text{End}_A(M)$ containing $\F$.
If $A^{\K}$ is  left noetherian, then $\K$ is a finitely generated field extension of $\F$.
\end{lemma}

We now return to the proof of Proposition \ref{finite}.

\begin{proof}[ Proof of Proposition \ref{finite}]Let $M$ be a simple $A$-module and
$D=\text{End}_A(M)$. Then $D$ is a division ring.  Let $\K$ be the center of $D$, and let $\mathbb{E}$ be a maximal subfield of $D$.

Note that $\F \leq \K \leq \mathbb{E}.$ By Lemma \ref{Noteherian2}, $A^\E$ is left noetherian.
Since $M$ is a simple $A$-module, it is finitely generated. Then by Lemma \ref{f-extension}, $\mathbb{E}$ is a finitely-generated field extension of $\F$.
Moreover, every element in $D$ is algebraic over $\F$ by Quillen's theorem, then $\mathbb{E}/\F$ is a finitely generated algebraic extension, hence $\E/\F$ is a finite extension.
In particular, $\text{dim}_{\K}(\mathbb{E}) < \infty$ and $\text{dim}_{\F}(\K) < \infty$.
By Lemma \ref{division-ring}, we have $\text{dim}_{\K}(D) < \infty$. Consequently, $\text{dim}_{\F}(D) < \infty$, and the proof is complete.
\end{proof}

\section{Strong endomorphism property of $A(V)$}
Using the results from previous sections, we can prove that Zhu's algebra $A(V)$ of a strongly finitely generated vertex operator algebra $V$ over an arbitrary field $\F$ satisfies the strong endomorphism property.

To proof this theorem, we first review the relations between the commutative algebras $R(V)$ and $\text{gr}A(V)$ studied in \cite{A,L1,L2,Liu,Z}.


Let $V$ be a vertex operator algebra over a field $\F$.
Recall that $R(V)=V/C_2(V)$ is a commutative Poisson algebra (cf. \cite{Z}), with the identity element $\bar{{\bf 1}}$, under the product and Lie bracket:
\begin{equation}\label{5.1}
\overline{u}\cdot \overline{v}=\overline{u_{-1}v} ,\qquad \{\bar{u},\bar{v}\}=\overline{u_0v},\quad \text{for}\ u,v\in V,
\end{equation}
where $\overline{u}$ represents the image of $u$ in $R(V)$.

For any $n \in \Z_+$, we define
$F_n(A(V))=A(V)_n=(V_0+ \cdots +V_n+O(V))/O(V) \subseteq A(V)$(cf. \cite{Z}).
It is easy to verify that  $(A(V)_n)_{n=0}^{\infty}$ is a filtration of $A(V)$.
Hence $A(V)$ is a filtrated algebra. Moreover, by the following property of $A(V)$:
$$u\ast v-v\ast u\equiv \sum_{i\geq 0}\binom{\wt u}{i} u_iv \pmod{O(V)},$$
 we can show that $\text{gr}A(V)$  is also a commutative Poisson algebra \cite{A, Liu}. Define a linear map $\phi : R(V) \rightarrow \text{gr}A(V)$ by letting:
$$\phi(\bar{u}):=u+A(V)_{n-1} \in A(V)_n/A(V)_{n-1},$$ for $n\in \Z_+$ and $u \in V_n$.
Then $\phi$ is a well-defined epimorphism between the commutative Poisson algebras $R(V)$ and $\text{gr}A(V)$ (see \cite{CJL,Liu}).

The following lemma is well-known (see \cite{A,L1,L2}). We include the proof of it for the completeness.

\begin{lemma} \label{RV}
Let $V$ be a strongly finitely generated vertex operator algebra over $\F$. Then:
\begin{enumerate}[{(1)}]
\item $R(V)$ is a finitely generated commutative algebra.
In particular, $R(V)$ is  noetherian.
\item $A(V)$ is a finite-type filtered algebra.
\end{enumerate}
\end{lemma}

\begin{proof}

(1) Let $U$ be a finite set that strongly generates $V$. Since
$$a_{-n}b_{-2}c=b_{-2}a_{-n}c+\sum_{j\geq 0}\binom{-n}{j}(a_jb)_{-n-2-j}c\equiv b_{-2}a_{-n}c\pmod{C_{2}(V)},$$
for any $a,b,c\in V$ and $n \geq 0$. It follows that for any spanning element $u^{1}_{-n_{1}}\dots u^{r}_{-n_{r}}{\bf 1}$ of $V$, where $u^i\in U$ and $n_{i}\geq 1$ for all $i$, if there exists some $i$ such that $n_{i}\geq 2$, then $u^{1}_{-n_{1}}\dots u^{r}_{-n_{r}}{\bf 1}\in C_{2}(V)$. Then by \eqref{5.1}, we have $R(V)$ is generated by $(U+C_{2}(V))/C_{2}(V)$, and so $R(V)$ is noetherian by Hilbert' basis theorem.

(2) Since $\phi$ is an epimorphism between the commutative algebras $R(V)$ and $\text{gr}A(V)$,
it follows that $\text{gr}A(V)$ as a commutative algebra is also finitely generated, and so $A(V)$ is a  finite-type filtered algebra.
\end{proof}

Now we prove the main theorem of this section:

\begin{theorem} \label{fi-end}
If $V$ is a strongly finitely generated vertex operator algebra over a field $\F$, and $\K /\F$ is a field extension, then
$A(V)^{\K}$ has the strong endomorphism property. In particular, $A(V)$ has the strong endomorphism property.
\end{theorem}
\begin{proof}
Because $V$ is  strongly finitely generated, $V^{\K}$ is also strongly finitely generated by Proposition \ref{exstf}.
This implies that $A(V)^{\K} \cong A(V^{\K})$ is a finite-type filtered algebra by Proposition \ref{Anvex} and Lemma \ref{RV}.
Now it follows from Proposition \ref{finite} that $A(V)^{\K}$ has the strong endomorphism property.
\end{proof}

\section{The endomorphism property of vertex operator algebras}
In this section, we will  prove the main results of the paper: Theorem \ref{main1} and Corollary \ref{Quillen-voa}.

\begin{lemma} \label{Anvapp1}
Let $V$ be a vertex operator algebra over a field $\F$. Then the following conditions are equivalent:
\begin{enumerate}[{(1)}]
\item  $A(V)$ has the endomorphism property.
\item  $V$ has the endomorphism property.
\item For any $n \in \Z_+$, $A_n(V)$ has the endomorphism property.
\end{enumerate}
\end{lemma}

\begin{proof}
(1) $\Rightarrow$ (2). Let $M$ be any irreducible admissible $V$-module over $\F$. Note that $M(0)$ is a simple $A(V)$-module.
By Corollary \ref{Anv3}, we have $\text{End}_V(M) \cong \text{End}_{A(V)}(M(0))$.
Since $A(V)$ has the  endomorphism property, every element in $\text{End}_V(M)$ is algebraic over $\F$. Thus, $V$ has the endomorphism property as well.

(2) $\Rightarrow$ (3).
Let $U$ be a simple $A_n(V)$-module. By (3) of Theorem \ref{Anv2}, $L_n(U)$ is an irreducible admissible $V$-module.
By Corollary \ref{Anv3} again, we have $\text{End}_{V}(L_n(U)) \cong \text{End}_{A_n(V)}(U)$.
Hence each element in $\text{End}_{A_n(V)}(U)$ is algebraic over $\F$.

(3) $\Rightarrow$ (1). Let $U$ be a simple $A(V)$-module.
By Theorem \ref{Anv1}, we know that $A(V)$ is a quotient algebra of $A_n(V)$,
hence $U$ is also a simple $A_n(V)$-module and $\text{End}_{A(V)}(U) \cong \text{End}_{A_n(V)}(U)$.
Hence each element in $\text{End}_{A(V)}(U)$ is algebraic over $\F$.
\end{proof}

By adopting a similar argument, we can prove the equivalency of the strong endomorphism properties:

\begin{lemma} \label{Anvapp2}
Let $V$ be a vertex operator algebra over a field $\F$. Then the following conditions are equivalent:
\begin{enumerate}[{(1)}]
\item  $A(V)$ has the strong endomorphism property.
\item $V$ has the strong endomorphism property.
\item For any $n \in \Z_+$, $A_n(V)$ has the strong endomorphism property.
\end{enumerate}
\end{lemma}

Before we prove the main theorem, we give some examples of vertex operator algebras with the strong endomorphism property.
Note that the vertex operator algebras in the following example are all strongly finitely generated, but none of them are $C_2$-cofinite.

\begin{example}
Let $\F$ be an arbitrary field with $\text{Char}(\F)=0$.

(1) If $V=\bar{V}_{vir}(c,0)$ is the (universal) Virasoro vertex operator algebra over $\F$ with the central charge $c$.
It is well-known that  $A(V) \cong \F[x]$ (see \cite{W}).
Since $\F[x]$ is a principal ideal domain, every simple $\F[x]$-module is finite-dimensional, then it follows that $A(V)$ has the strong endomorphism property, and so $\bar{V}_{vir}(c,0)$ also has the strong endomorphism property by Lemma \ref{Anvapp2}.

(2) If $V=M_{\mathfrak{\hat{h}}}(1,0)$ is the  level one Heisenberg vertex operator algebra (\cite{FZ}) over $\F$
associated with the $n$-dimensional vector space $\mathfrak{h}$, then $A(V) \cong \F[x_1, \cdots, x_n]$.
Let $S$ be any simple $\F[x_1, \cdots, x_n]$-module. Then $S \cong \F[x_1, \cdots, x_n]/J$, for some maximal ideal of $\F[x_1, \cdots, x_n]$.
Clearly, $\F[x_1, \cdots, x_n]/J$ is a field extension of $\F$, and it is a finitely generated $\F$-algebra.
Hence $\F[x_1, \cdots, x_n]/J$ is a finte algebraic extension of $\F$ by Hilbert's Nullstellensatz.
Consequently, every simple $A(V)$-module is finite-dimensional.
Hence, $M_{\mathfrak{\hat{h}}}(1,0)$ has the strong endomorphism property  by Lemma \ref{Anvapp2}.

(3) Let $\mathfrak{g}$ be a finite-dimensional simple Lie algebra over a field $\F$,
equipped with a nondegenerate symmetric invariant bilinear form $\langle \cdot , \cdot \rangle$.
Assume that the Casimir operator $\Omega$ acts on $\mathfrak{g}$ as a scalar $2h \in \F$, with $k \neq -h$.
Let $V=V_{\mathfrak{\hat{g}}}(k,0)$ be the level $k $ vacuum module vertex operator algebra over $\F$
 associated with the finite-dimensional simple Lie algebra $\mathfrak{g}$ (\cite{FZ,LL}).
 Then $A(V) \cong U(\mathfrak{g})$, the universal enveloping algebra of $\mathfrak{g}$ (see \cite{FZ}).
 It is well-known that $U(\mathfrak{g})$ has the strong endomorphism property (\cite{D,MR}).
  Hence $V_{\mathfrak{\hat{g}}}(k,0)$ has the strong endomorphism property by Lemma \ref{Anvapp2}.

\end{example}

\begin{theorem} \label{main1} Let $V$ be a vertex operator algebra over a field $\F$.
	\begin{enumerate}[{(1)}]

\item  If $\F$ is an uncountable field, then $V$ has the endomorphism property.

\item  If $V$ is $C_2$-cofinite, then $V$ has the strong endomorphism property.

\item  If $\F=\R$ the real number field, then $V$ has the strong endomorphism property.

\item (Quillen's theorem) If $V$ is strongly finitely generated over $\F$, then $V$ has the strong endomorphism property.
\end{enumerate}
\end{theorem}

\begin{proof}
(1) By Lemma \ref{Anvapp1}, it is enough to prove that $A(V)$ has the  endomorphism property.
Let $S$ be a simple $A(V)$-module. Recall that $V$ is of countable dimension over $\F$. Then $A(V)$, as a quotient space of $V$, also has countable dimension. Thus $S$  has countable dimension over $\F$ as well.

 Let $\theta \in \text{End}_{A(V)}(S)$, we adopt a similar argument in \cite{Q} to show that $\theta$ is algebraic. Suppose $\theta$ is not algebraic, then $S$ is a nonzero vector space over $\F(\theta)$, and $\text{dim}_{\F}(S)\geq \text{dim}_{\F}(\F(\theta)) \geq \text{card}~\F$,
where the last inequality follows from the fact that
the elements of the form $(\theta-\lambda)^{-1} \ (\lambda \in \F)$ in $\F(\theta)$ are linearly independent.
This contradicts the assumption that $\F$ is uncountable. Hence $\theta$ must be algebraic.

(2) If $V$ is $C_2$-cofinite, then by the epimorphism $\phi : R(V)\rightarrow \gr(A(V))$ in Section 5 we know that $\gr(A(V))$, as well as $A(V)$, are finite-dimensional.
Naturally, $A(V)$ has the strong endomorphism property. Then by Lemma \ref{Anvapp2}, $V$ also has the strong endomorphism property.

(3) By Lemma \ref{Anvapp2} again, it is enough to prove that $A(V)$ has the strong endomorphism property.
Let $S$ be any simple $A(V)$-module, and let $D=\text{End}_{A(V)}(S)$.
Since $\R$ is an uncountable field, every element in $D$ is algebraic over $\R$ by (1).
Consequently, $D$ is a division algebra over $\R$ such that every element in $D$ is algebraic over $\R$.
Hence $D$ must be isomorphic to $\R, \C$ or $\mathbb{H}$, where $\mathbb{H}$ is the  quaternion algebra (see \cite{H}).
This implies that $\text{dim}_{\R}(D) < \infty$, as desired.

(4) Since $V$ is  strongly finitely generated, it follows from  Theorem \ref{fi-end} that $A(V)$ has the strong endomorphism property. So does $V$ by Lemma \ref{Anvapp2}.
\end{proof}

\begin{remark}
The $C_2$-cofinite CFT-type vertex operator algebras over the complex number field $\C$ are all strongly finitely generated (see \cite{L1}). So part (2) in Theorem \ref{main1} is a special case of part (4) when $\F=\C$ and $V$ is of the CFT-type. 
\end{remark}

The following corollary is an immediate consequence of Lemma \ref{Anvapp2} and Theorem \ref{main1}:

\begin{corollary} \label{main1cor}
(1) If $V$ is a vertex operator algebra over the real number field $\R$, then $A_n(V)$ has the strong endomorphism property.

(2) If $V$ is a strongly finitely generated vertex operator algebra over an arbitrary field $\F$, then $A_n(V)$ has the strong endomorphism property.

\end{corollary}

\begin{remark}
If $V$ is a strongly finitely generated vertex operator algebra over a field $\F$,
by Lemma \ref{Noteherian2} and Lemma  \ref{RV}(2), we know that $A(V)$ is stably noetherian.
Moreover, by Corollary \ref{main1cor}, $A_n(V)$ has the strong endomorphism property for any $n \in \Z_+$.
A natural question is whether $A_{n}(V)$ is also stably noetherian.
\end{remark}

\begin{proposition} \label{rational-1}
Let $V$ be a vertex operator algebra over a field $\F$, with the strong endomorphism property. Then
\begin{enumerate}[{(1)}]
\item If $V$ is rational, then  $A_n(V)$ is a finite dimensional semi-simple associative algebra for each $n\in \Z_{+}$.
\item If $V$ is simple, then $V$ is rational if and only if  $A_n(V)$ is finite-dimensional semi-simple for each $n \in \Z_+$.
\end{enumerate}
\end{proposition}

\begin{proof}
(1) If $V$ is rational, then by Theorem \ref{Anv1} (4), $A_n(V)$ is a semi-simple left $A_n(V)$-module for each $n\in \Z_+$.
Thus $A_n(V)$ is an aritinian algebra.
Let $\{ S_1, \cdots , S_k \}$ be the set of the isomorphism classes of simple $A_n(V)$-modules, then by Artin-Wedderburn theorem,
$$A_n(V)^{op} \cong \bigoplus_{i=1}^k \text{M}_{n_i}(D_i) \ \text{for some} \ n_i \in \Z_+ , \ \text{where} \ D_i=\text{End}_{A_n(V)}(S_i), \ \text{for\ all\ } i.$$
Since $V$ has the strong endomorphism property, we have $\text{dim}_{\F}(D_i) < \infty$ for all $i$.
Consequently,  $A_n(V)$ is a finite dimensional semi-simple associative algebra for each $n\in \Z_+$.

(2) follows immediately from (1) and Theorem \ref{Anv1} (5).
\end{proof}

\begin{proposition} \label{end-pro}
Let $\F$ be an algebraically closed field, and let $V$ be a vertex operator algebra over $\F$, with the  endomorphism property. Let $M=\bigoplus_{n \in \Z_+}M(n)$ be an irreducible admissible $V$-module. Then
\begin{enumerate}[{(1)}]
\item  $\text{End}_V(M) \cong \F$. Moreover, there exists a $\lambda \in \F$ such that $L(0)|_{M(n)}=\lambda +n$ for all $n$.


\item If $V$ is rational, then  $A_n(V)$ is a finite dimensional semi-simple associative algebra for every $n\in \Z_+$.
\end{enumerate}
\end{proposition}

\begin{proof}
The proof of (1) is straightforward (see also \cite[Proposition 4.5.5]{LL}). It follows from (1) that $V$ has the strong endomorphism property.
Hence (2) is an immediate consequence of Proposition \ref{rational-1}.
\end{proof}


By the previous results, Theorem \ref{main1} (4), Proposition \ref{rational-1}, and Proposition \ref{end-pro}, we have the following: 

\begin{corollary} \label{Quillen-voa}
Let $V$ be a strongly finitely generated vertex operator algebra over a field $\F$, and let $M=\bigoplus_{n \in \Z_+}M(n)$ be an irreducible admissible $V$-module. Then we have:
\begin{enumerate}[{(1)}]
\item (Quillen's theorem) $\text{End}_V(M) < \infty$.

\item (Schur's lemma) If $\F$ is an algebraic closed field, then $\text{End}_V(M) \cong \F$.
Moreover, there exists a $\lambda \in \F$ such that $L(0)|_{M(n)}=\lambda +n$ for all $n$.


\item  If $V$ is rational, then  $A_n(V)$ is a finite-dimensional semi-simple associative algebra for any $n \geq 0$.
\end{enumerate}
\end{corollary}

\begin{remark}
Let $V$ be a vertex operator algebra over an arbitrary field $\F$.
It is well-known that if $\F$ is uncountable and algebraically closed, then Schur's Lemma in Corollary \ref{Quillen-voa} (2) is satisfied.
In \cite{DR1}, the authors made the conjecture that the assumption $\F$ is uncountable in Schur's Lemma is unnecessary.

Now by Corollary \ref{Quillen-voa}, we can drop the uncountable assumption for $\F$ in Schur's Lemma for strongly finitely generated vertex operator algebra $V$.
\end{remark}


\section{Applications: Scalar Extensions}

In \cite{M}, the author studied the vertex operator algebras $V$ over $\R$, and proved that $V^\C=\C\otimes_\R V$ is rational if $V$ is rational over $\R$. As an application of Proposition \ref{rational-1}, we will prove that the inverse is also true.

\begin{theorem} \label{exrationa}
Let $\K/ \F$ be a finite separable field extension, and $V$ be a simple vertex operator algebra over $\F$ s.t. $V_0=\F {\bf 1}$.
Assume that $V$ and $V^{\K}$ both have the strong endomorphism property. Then $V$ is rational if and only if $V^{\K}$ is rational.
\end{theorem}

\begin{proof}
The ``if" part: Since $V^{\K}$ is rational and has the strong endomorphism property, then by Proposition \ref{rational-1}, $A_n(V^{\K}) \cong \K \otimes A_n(V)$ is a finite-dimensional
semi-simple associative algebra for any $n \in \Z_+$.
Since $\text{dim}_{\F}(A_n(V))=\text{dim}_{\K}(\K \otimes A_n(V))$, we have $\text{dim}_{\F}(A_n(V)) < \infty$.
We claim that $A_n(V)$ is also semi-simple for each $n \in \Z_+$. Indeed, we just need to show that $A_n(V)$ has no nonzero nilpotent ideal since it is already finite-dimensional.
Suppose that $A_n(V)$ has a nonzero nilpotent ideal $I$, then $\K \otimes I$ is a nonzero nilpotent ideal of $A_n(V^{\K})$, which contradicts the fact that $A_n(V^{\K})$ is semi-simple. Hence each $A_n(V)$ is semi-simple, and so $V$ is rational by Theorem \ref{Anv1}(5).

The ``only if" part: If $V$ is rational, then by Proposition \ref{rational-1} again, $A_n(V)$ is a finite-dimensional semi-simple associative algebra for each $n \in \Z_+$.
Since $\K/\F$ is a finite separable extension, it follows from Lemma \ref{ext-semi} that $A_n(V^{\K}) \cong \K \otimes A_n(V)$ is also finite dimensional semi-simple over $\K$.
Moreover, $V^{\K}$ is a simple vertex operator algebra over $\K$ by Corollary \ref{ext-simple2}. Hence  $V^{\K}$ is rational by Theorem \ref{Anv1}(5).
\end{proof}

\begin{remark}
When $\K/ \F$ is a finite separable extension, we believe that $V$ has the strong endomorphism property if and only if $V^{\K}$ has the strong endomorphism property.
But we do not know how to prove it.
\end{remark}

\begin{corollary}  \label{Rrational}
Let $V$ be a simple vertex operator algebra over the real number field $\R$, with $V_0=\R{\bf 1}$.
Then $V$ is rational if and only if $V^{\C}$ is rational.
\end{corollary}
\begin{proof}
	By Theorem \ref{main1} (3), any vertex operator algebra $V$ over $\R$ has the strong endomorphism property, and by Proposition \ref{end-pro} (1), $V^\C$ also has the strong endomorphism property. Moreover, the extension $\C/\R$ is clearly finite separable. Then by Theorem \ref{exrationa}, we have $V$ is rational if and only if $V^\C$ is rational.
	\end{proof}

\begin{corollary}
Let $\K/ \F$ be a finite separable extension, and $V$ be a simple strongly finitely generated vertex operator algebra over a field $\F$ such that $V_0=\F {\bf 1}$.
Then $V$ is rational if and only if $V^{\K}$ is rational.
\end{corollary}
\begin{proof}
Note that $V^{\K}$ is also strongly finitely generated by Proposition \ref{exstf}.
By Corollary \ref{Quillen-voa} (1), both $V$ and $V^\K$ have the strong endomorphism property. Now the conclusion follows from Theorem \ref{exrationa}.
\end{proof}

\begin{theorem} \label{absimple}
Let $V$ be a strongly finitely generated vertex operator algebra over a field $\F$, and $M$ be an irreducible admissible  $V$-module.
Then the following conditions are equivalent:
\begin{enumerate}[{(1)}]

\item  $M$ is an absolutely irreducible $V$-module.

\item  $M(0)$ is an absolutely simple $A(V)$-module.

\item $\text{End}_V(M)=\text{End}_{A(V)}(M(0))=\F.$
\end{enumerate}
\end{theorem}

\begin{proof}
By Lemma \ref{lm3.6} and Corollary \ref{Anv3},
it is enough to prove that $M(0)$ is absolutely simple if and only if $\text{End}_{A(V)}(M(0))=\F.$
The following proof is similar to  \cite[Proposition 2.6.5]{D}.

Assume that $\text{End}_{A(V)}(M(0))=\F$.
We claim that $M(0)^{\K}$ is a simple $A(V)^{\K}$-module for any field extension $\K / \F$.
Indeed, let $N$ be a nonzero submodule of $M(0)^{\K}$.
Take a nonzero element $k_1 \otimes m_1+\cdots + k_p \otimes m_p \in N \ (k_i \in \K , m_i \in M(0),  p>0)$
such that $ m_1, \dots , m_p $ are linearly independent over $\F=\End_{A(V)}(M(0))$.
Since $M(0)$ is a simple $A(V)$-module,
by the Jacobson's density theorem, there exists some $v \in A(V)$ such that $[v]\cdot m_1=o(v)m_1=m_1$ and $[v] \cdot m_i=o(v)m_i=0$ for any $i \neq 1$.
Hence $o(k_1^{-1} \otimes v)(k_1 \otimes m_1+\cdots + k_p \otimes m_p)=1 \otimes m_1 \in N.$
Since $M(0)$ is a simple $V$-module, we have $M(0)=\{o(v)m_1 ~ | ~ v \in V\}.$
Therefore, we have $ M(0)^{\K} =N$, and $M(0)^{\K}$ is a simple $A(V)^{\K}$-module.

Now we assume that $M(0)$ is absolutely simple. For any  $f \in \text{End}_{A(V)}(M(0)),$
we need to show that $f= \lambda Id_{M(0)}$ for some $\lambda \in F$. Let $\overline{\F}$ be an algebraic closure of $\F$.
Then $\text{End}_{A(V)^{\K}}(M(0)^{\K})=\overline{\F}$ by Corollary \ref{Quillen-voa}.
On the other hand, $1 \otimes f \in \text{End}_{A(V)^{\K}}(M(0)^{\K}),$
so $1 \otimes f= \lambda \otimes Id_{M(0)}$ for some $\lambda \in \overline{\F}.$
This implies that $\lambda \in \F$ and $f=\lambda Id_{M(0)}.$ The proof is complete.
\end{proof}

\begin{remark}
The Theorem \ref{absimple} is useful, especially because it provides an intrinsic test for a module to be absolutely simple, which
involves only the computation of $\text{End}_V(M)$ or $\text{End}_{A(V)}(M(0))$, rather than the consideration of extension fields.
\end{remark}

\begin{corollary}
Let $V$ be a strongly finitely generated vertex operator algebra  over an algebraically closed field $\F$, and let $M$ be a simple admissible $V$-module.
Then $M$ is absolutely simple.

\end{corollary}


\begin{thebibliography}{}

\bibitem[A]{A}	
 T. Arakawa, A remark on the $C_2$-cofiniteness condition on vertex algebras, {\em Math. Z.} {\bf 270} (2012), no. 1-2, 559-575.


\bibitem[B]{B}J. P. Bell, Noetherian algebras over algebraically closed fields,
{\em J.Algebra} {\bf 310}, 2007, 148-155.


\bibitem[CJL]{CJL} A. Caradot, C. Jiang, and Z. Lin, {\em Yoneda algebras of the triplet vertex operator algebra}, arXiv:2204.01650.

\bibitem[CL]{CL}
Creutzig, T., A. Linshaw, Orbifolds of symplectic fermion algebras,
{\em Trans. Am. Math. Soc.} {\bf 369(1)}, 467-494 (2017).

\bibitem[D]{D} J. Dixmier,  \emph{Enveloping Algebras}, (The 1996 Printing of the 1977 English Translation)
Graduate Studies in Mathematics, Vol. II, American Mathematical Society, 1996.

\bibitem[DS]{DS}
G. Dospinescu, B. Schraen, Endomorphism algebras of admissible $p$-adic
representations of $p$-adic Lie groups, {\em Represent. Theory} {\bf 17} (2013), 237-246.

\bibitem[DJ]{DJ} C. Dong, C. Jiang, Bimodules associated to vertex operator algebras, {\em Math. Z. } {\bf 289} (2008), 799-826.

\bibitem[DLM]{DLM}
C. Dong, H. Li, and G. Mason,
Vertex operator algebras and associative algebras,
{\em J. Algebra.} {\bf 206} (1998), 67--96.

\bibitem[DLM1]{DLM1}
C. Dong, H. Li, and G. Mason,
Twisted representations of vertex operator algebras,
{\em Math. Ann.} {\bf 310} (1998), 571--600.





\bibitem[DR1]{DR1}
C. Dong, L. Ren, Representations of vertex operator algebras over an arbitrary field,
{\em J.Algebra} {\bf 403} (2014), 497-516.

\bibitem[DR2]{DR2}
C. Dong, L. Ren, Vertex operator algebras associated to the Virasoro algebra over an arbitrary field,
{\em Trans. Amer. Math. Soc.} {\bf 368} (2016), no. 7, 5177-5196.


\bibitem[FZ]{FZ}I. Frenkel, Y. Zhu, Vertex operator algebras associated to representations of affine and Virasoro algebras,
{\em Duke Math. J.} {\bf 66} (1992), 123-168.

\bibitem[JLM]{JLM}
X. Jiao, H. Li, Q. Mu,  Modular Virasoro vertex algebras and affine vertex algebras,  {\em J. Algebra} {\bf 519} (2019), 273-311.


\bibitem[L1]{L1}
H. Li, Abelianizing vertex algebras, {\em Comm. Math. Phys.}
	{\bf 259}, 391-411 (2005).

\bibitem[L2]{L2}
H. Li, Some finiteness properties of regular vertex operator algebras, \emph{J. Algebra} {\bf212}(1999), 495-514

\bibitem[LA]{LA} A. Linshaw, Invariant subalgebras of affine vertex algebras, {\em Adv. Math.} {\bf 234} (2013), 61-84.


\bibitem[Liu]{Liu} J. Liu, On Filtrations of A(V), arXiv:2103.08090.


\bibitem[LL]{LL}
J. Lepowsky, H. Li, Introduction to Vertex Operator Algebras and Their Representations,
{\em Progress in Mathematics}, Vol. {\bf 227}, Springer, 2004.
	


\bibitem[LM1]{LM1}	
H, Li; Q, Mu, Heisenberg VOAs over fields of prime characteristic and their representations,
{\em Trans. Amer. Math. Soc.} {\bf 370} (2018), no. 2, 1159-1184.

\bibitem[LM2]{LM2}
H, Li; Q, Mu, Symmetric invariant bilinear forms on modular vertex algebras,
{\em J. Algebra} {\bf 513} (2018), 435-465.	

\bibitem[Lang]{Lang}S. Lang, Algebra, 3rd ed., Vol {\bf 211}, Springer-Verlag, 2002.

\bibitem[MR]{MR} J.C. McConnell and J.C. Robson, Noncommutative Noetherian Rings,
Graduate Studies in Mathematics, Vol. {\bf 30}, American Mathematical Society, (2001).

\bibitem[M]{M} M. Miyamoto, A new construction of the Moonshine vertex operator algebra over the real number field,
{\em Ann. of Math.} (2) {\bf 159} (2004), 535-596.



\bibitem[Q]{Q}	
D. Quillen, On the endomorphism ring of a simple module over an enveloping algebra.
{\em Proc. Amer. Math. Soc.} {\bf 21} (1969), 171-172.
	
\bibitem[R]{R}	
L. Ren, Modular $A_n(V)$ theory. {\em J. Algebra.} {\bf 485} (2017), 254-268.

\bibitem[H]{H}
 T. W. Hungerford, Algebra, {\em Graduate Texts in Mathematics,} vol. {\bf 73}. Springer, Berlin (2003).


\bibitem[W]{W}
W. Wang. Rationality of Virasoro Vertex Operator Algebra, {\em Int. Math. Res. Not.} {\bf 7} (1993), 197-211.


\bibitem[Z]{Z} Y. Zhu; \emph{Modular invariance of characters of
vertex operator algebras.} J. Amer. Math. Soc. {\bf 9} 1996., 237-302.
\end{thebibliography}
\end{document}